\documentclass{amsart}

\usepackage{graphicx}
\usepackage{tikz-cd}
\usepackage{amssymb}
\usepackage{amsthm}
\usepackage[normalem]{ulem}
\usetikzlibrary{arrows}

\theoremstyle{plain}
\newtheorem{theorem}{Theorem}[section]
\newtheorem{lemma}[theorem]{Lemma}
\newtheorem{proposition}[theorem]{Propostion}
\newtheorem{corollary}[theorem]{Corollary}
\newtheorem*{corollary*}{Corollary}

\theoremstyle{definition}

\theoremstyle{remark}
\newtheorem{remark}[theorem]{Remark}

\numberwithin{equation}{section}

\begin{document}

\title[Elements Attaining Norm in a Finite-Dimensional Representation]{Elements of $C^*$-algebras Attaining Their Norm in a Finite-Dimensional Representation}

\author{Kristin Courtney}
\address{Department of Mathematics, University of Virginia, Charlottesville, VA}
\email{kc2ea@virginia.edu}

\author{Tatiana Shulman}
\address{Department of Mathematical Physics and Differential Geometry, Institute of Mathematics of Polish Academy of Sciences, Warsaw}
\thanks{The research of the
second-named author was supported by the Polish National Science Centre grant
under the contract number DEC- 2012/06/A/ST1/00256 and from the Eric Nordgren
Research Fellowship Fund at the University of New Hampshire.}

\subjclass[2010]{46L05; 46L67}



\keywords{AF-telescopes, RFD, Projective}

\begin{abstract}
We characterize the class of RFD $C^*$-algebras as those containing a dense subset of elements that attain their norm under a finite-dimensional representation. We show further that this subset is the whole space precisely when every irreducible representation of the $C^*$-algebra is finite-dimensional, which is equivalent to the $C^*$-algebra having no simple infinite-dimensional AF subquotient.
We apply techniques from this proof to show the existence of elements in more general classes of $C^*$-algebras whose norms in finite-dimensional representations fit certain prescribed properties.
 \end{abstract}

\maketitle

\section{Introduction}

Information about finite-dimensional representations of a $C^*$-algebra is useful for studying its structural properties. RFD $C^*$-algebras are those which have many finite-dimensional representations. 
Recall that a $C^*$-algebra is called {\it residually finite-dimensional} (RFD) if it has a separating family of finite-dimensional representations.

One of the first results on RFD $C^*$-algebras, due to Choi (\cite{C}), is the fact that the full $C^*$-algebra $C^*(\mathbb{F}_n)$ of the free group  is RFD (\cite{C}). In the ensuing years, various characterizations of RFD $C^*$-algebras have been obtained (notably in \cite{EL, Ar, Don}), and various classes of $C^*$-algebras were proved to be RFD. A notable class of RFD $C^*$-algebras are those whose irreducible representations are all finite-dimensional. We call such $C^*$-algebras {\it FDI} for ``Finite-Dimensional Irreps''. This class includes, in particular, ($n$-)subhomogeneous $C^*$-algebras.

  Examples of RFD $C^*$-algebras arising from groups include full group $C^*$-algebras of amenable maximally periodic groups \cite{Bekka}, surface groups and fundamental groups of closed hyperbolic 3-manifolds that fiber over the circle \cite{Shalom}, and many 1-relator groups with non-trivial center \cite{HadwinShulman}. Other classes of RFD $C^*$-algebras include amalgamated products of commutative $C^*$-algebras \cite{Korchagin}, projective $C^*$-algebras \cite{L}, universal  $C^*$-algebras of algebraic elements \cite{LoringShulman},  the soft torus $C^*$-algebra \cite{EE}, and certain just-infinite $C^*$-algebras \cite{GMR}. (This list is certainly incomplete.)
The class of RFD $C^*$-algebras is also closed under free products \cite{EL} (see also  \cite{GM}), minimal tensor products \cite{BO}, extensions, and subalgebras.

In \cite{FNT} Fritz, Netzer and Thom  proved that every element in the group algebra $\mathbb{CF}_n$ attains its universal norm under some finite-dimensional unitary representation. Viewing $\mathbb{CF}_n$ as a dense subalgebra of $C^*(\mathbb{F}_n)$, it is natural to ask whether there exists in other RFD $C^*$-algebras a dense subset of elements that attain their norm under a finite-dimensional representation. In section \ref{section3}, we prove that this is indeed true. Moreover, this characterizes RFD $C^*$-algebras (Corollary \ref{RFDcor}).

Looking at the result of Fritz, Netzer and Thom, one can ask further questions. For instance, are there elements in $C^*(\mathbb{F}_n)$ other than the elements of $\mathbb{CF}_n$ that attain their norm under a finite-dimensional representation? Could this be true for all elements?

In Section \ref{section4}, we prove that all elements of a $C^*$-algebra attain their norm under a finite-dimensional representation if and only if the $C^*$-algebra has no infinite-dimensional irreducible representation, i.e the $C^*$-algebra is FDI (Theorem \ref{FDI}). In particular, this implies the existence of elements in $C^*(\mathbb{F}_n)$ that do not attain their norm under a finite-dimensional representation. Moreover, we show that $A$ is FDI if and only if $A$ has no $C^*$-subalgebra which surjects onto some simple, infinite-dimensional AF-algebra.

In Section \ref{section5},  we introduce  seminorms associated with finite-dimensional representations and study their growth. Namely, for a $C^*$-algebra $A$ with at least one irreducible representation of dimension no larger than $k<\infty$, we define a seminorm $\|\cdot\|_{\mathbb{M}_k}$ at each $a\in A$ by
\begin{align*}
    \|a\|_{\mathbb{M}_k}=\sup\{\|\pi(a)\|\},
\end{align*}
 where the supremum is taken over all irreducible representations of dimension no larger than $k$. If $A$ has irreducible representations of dimensions $n_1<n_2<...<\infty$, \footnote{Our results, when relevant, also hold for $C^*$-algebras for which this sequence is finite.}
 then for each $a\in A$, we have a non-decreasing sequence
$$(\|a\|_{\mathbb{M}_{n_k}})_{k\in \mathbb{N}}.$$ Let $\Lambda(A)$ be the set of all such sequences. We want to know what sequences can be found in $\Lambda(A)$ for a given $C^*$-algebra $A$. 
In Theorem \ref{1} we prove that $\Lambda(A)$ contains the set of all nondecreasing sequences of positive numbers which are eventually constant. 
 We show that those two sets coincide exactly when $A$ is FDI (Corollary \ref{3}). When $A$ is RFD but not FDI, we describe the behavior of some sequences in $\Lambda(A)$ that are not eventually constant (Theorem \ref{RFDnotFDI}).

A technique developed in Section \ref{section4} allows us to say more about the subset of all elements which attain their norm under a finite-dimensional representation. Particularly, in Theorem \ref{subspace} we prove that this subset is additively closed if and only if it is multiplicatively closed if and only if the $C^*$-algebra is FDI. In particular, it implies that there exist elements in $C^*(\mathbb{F}_n)\backslash \mathbb {CF}_n$ that attain their norm under a finite-dimensional representation. 

One of our main tools in Sections \ref{section4} and \ref{section5} is the projectivity of AF-telescopes, discovered by Loring and Pedersen \cite{LP}. For a simple AF-algebra, it is often straightforward to find elements in its AF-telescope whose norms under finite-dimensional representations fit certain prescribed properties. When projectivity can be invoked, it can be sometimes used to lift said elements to elements in another $C^*$-algebra whose norms under finite-dimensional representations fit the same properties. 

 If we know an element in a $C^*$-algebra attains its norm in some finite-dimensional representation, it is natural to then ask for an upper bound for the dimension required to witness this. In their proof that any element of $\mathbb{CF}_n$  achieves its norm in a finite-dimensional representation, Fritz, Netzer, and Thom give an estimate of the dimension of such a representation (\cite{FNT}, Lemma 2.7). If $\ell$ is the length of the longest word in the support of the element, then such a  representation can be chosen of dimension no more than $4n^\ell$.  In Section \ref{section5}, we find a better bound on the dimension for binomials in $\mathbb{CF}_n$. In Theorem \ref{eigenvalue} we prove that for any nontrivial, balanced, reduced word $ w\in \mathbb{F}_n$ of length $\ell$ and any $\lambda\in \mathbb{T}$, there exists a representation $\pi:C^*(\mathbb{F}_n)\to \mathbb{M}_{2\ell}$ such that the spectrum of $\pi(w)$ contains $\lambda$. From this theorem  we deduce (Proposition \ref{dimestimate}) that any element of the form $\alpha w_1+\beta w_2$, where  $\alpha,\beta\in \mathbb{C}$ and $w_1,w_2\in\mathbb{F}_n$, attains its norm under a $2\ell$-dimensional representation of $C^*(\mathbb{F}_2)$; here $\ell$ is the length of the reduced word $w_2^{-1}w_1$.

We are grateful to David Sherman for many useful discussions.


\section{Preliminaries}\label{section1}
\subsection{AF Mapping Telescopes and Projective $C^*$-algebras}

We briefly introduce AF mapping telescopes (also called AF-telescopes); for more information, see \cite{L} or \cite{LP}. 

Let $A=\overline{\bigcup{A_n}}$ be an inductive limit of an increasing sequence of $C^\ast$-algebras $$A_1 \subset A_2 \subset \ldots \subset A$$
with injective connecting maps. We define the {\it mapping telescope} of $(A_n)$ as the $C^\ast$-algebra
\begin{align*}
    T(A)=\{f\in C_0((0,\infty],A)| f(t)\in A_{\lceil t \rceil} \ \forall t\in(0,\infty)\}
\end{align*}
where $\lceil t\rceil=\min\{n\in \mathbb{N}: n\geq t\}$. 
Obviously the mapping telescope depends on the sequence $(A_n)$, but we will use the notation $T(A)$ as opposed to $T(A_1,A_2,...)$ and specify the inductive sequence when necessary. In particular, we denote by $T(\mathbb{M}_{2^\infty})$ the mapping telescope corresponding to the inductive sequence 
\begin{align*}
    \mathbb{M}_2\subset\mathbb{M}_4\subset...\subset \mathbb{M}_{2^n}\subset...\subset \mathbb{M}_{2^\infty}
\end{align*}
where $\mathbb{M}_{2^n}$ is identified with a subalgebra of $\mathbb{M}_{2^{n+1}}$ by the map $a\mapsto a\oplus a$. Recall that $\mathbb{M}_{2^\infty}$ is referred to as the CAR algebra and is a simple $C^\ast$-algebra (\cite{Davidson}). We denote by $T(\mathbb{K}(\ell^2))$ the mapping telescope corresponding to the inductive sequence
$$\mathbb{C}\subseteq \mathbb{M}_2\subseteq ...\subseteq \mathbb{M}_n\subseteq...\subseteq \mathbb{K}(\ell^2)$$
with embeddings $a\mapsto a\oplus 0$.
When each $A_n$ is finite-dimensional, the $C^\ast$-algebra $A$ is AF and we call $T(A)$ an {\it AF-telescope}.

For the sake of consistency, we define the AF-telescope for inductive sequences of the form 
$$\mathbb{M}_{n_1}\subset  \mathbb{M}_{n_2} \subset  \ldots \subset \mathbb{M}_{n_N},$$
as
\begin{align*}
    T(\mathbb{M}_{n_N}):=\{f\in C_0((0,\infty],\mathbb{M}_{n_N})| f(t)\in \mathbb{M}_{n_{\lceil t \rceil}} \ \forall t\in(0,n_N]\}.
\end{align*}

Recall that a $C^*$-algebra $A$ is {\it projective} (\cite{B1}, \cite{L}) if given $C^\ast$-algebras $B$ and $C$ with surjective $\ast$-homomorphism $q:B\to C$, any $\ast$-homomorphism $\phi:A\to C$ lifts to a $\ast$-homomorphism $\psi:A\to B$ such that $q\circ\psi=\phi$. In other words, we have the following commutative diagram.
\[
\begin{tikzcd}
& B\arrow[two heads]{d}{q}\\
A\arrow{ur}{\psi}\arrow{r}[swap]{\phi} & C
\end{tikzcd}
\]

In \cite{LP}, Loring and Pedersen proved that all AF-telescopes are projective. This fact will be used repeatedly throughout the paper.

\subsection{Type I and GCR $C^\ast$-algebras}\label{GCR}
 In sections \ref{section3} and \ref{section4}, we rely on a result (Theorem \ref{non-GCR}) of Glimm and Sakai. The particular formulation we would like to cite is not so readily found in the literature, so we briefly describe it here.
 
Let $\mathcal{H}$ be a Hilbert space. 
A $C^\ast$-algebra $A$ is called {\it GCR} if $K(\mathcal{H})\subseteq \pi(A)$ for any irreducible representation $(\pi,\mathcal{H})$ of $A$. 
In particular, all FDI $C^*$-algebras are GCR. 
 It is due to a deep theorem of Glimm and Sakai that a $C^\ast$-algebra is GCR iff it is type I (see \cite{G} for the classic theorem and \cite{S2} for the nonseparable case). 
 We will call all such algebras GCR.

A $C^\ast$-algebra is {\it NGCR (antiliminal)} if it contains no nonzero abelian elements, i.e. there is no nonzero $ x\in A$ so that $\overline{x^*A_0x}$ is commutative. 
 Glimm (\cite{G}) and Sakai (\cite{S1}) have shown that an NGCR $C^\ast$-algebra must have a subquotient isomorphic to the CAR algebra, i.e. it has a subalgebra that surjects onto the CAR algebra. Since a GCR $C^*$-algebra is characterized as having no NGCR quotients (see \cite[Section IV.1.3]{B}), we arrive at the following formulation of the result.
 
\begin{theorem}[\cite{G}, \cite{S1}]\label{non-GCR}
Let $A$ be a $C^\ast$-algebra that is not GCR. Then $A$ has a subquotient isomorphic to the CAR algebra. 
\end{theorem}

\section{A Characterization of RFD $C^\ast$-algebras}\label{section3}
In this section, we characterize RFD $C^\ast$-algebras as being exactly those which have a dense subset of elements that attain their norm under a finite-dimensional representation. In fact, we prove that, for any residually class $\mathcal{C}$ $C^\ast$-algebra (i.e. an algebra with a separating family of representations which are class $\mathcal{C}$) the set of elements that attain their norm under a class $\mathcal{C}$ representations is dense.

First, we give a well-known characterization for a family of representations to be separating.

\begin{lemma}\label{sup} Let $A$ be a $C^\ast$-algebra and $\mathcal F$ be a separating family of its representations.
Then for each $a\in A$,
 $$\|a\|= \sup_{\pi\in \mathcal F} \|\pi(a)\|.$$
\end{lemma}
\begin{proof} Since $\mathcal F$ is separating, the representation $a\mapsto \oplus_{\pi\in \mathcal F} \pi(a)$ is injective. Hence it is isometric.
\end{proof}

\begin{theorem}\label{RFD}
Let $A$ be a $C^\ast$-algebra,  $\mathcal{F}$ a family of representations of $A$, and define
\begin{align*}
    A_\mathcal{F}:=\{a\in A\ |\ \|a\|=\max_{\pi\in \mathcal{F}}\|\pi(a)\|\}.
\end{align*}
Then the following are equivalent:
\begin{enumerate}
    \item[(i)] $A_\mathcal{F}$ is dense in $A$.
    \item[(ii)] $\mathcal{F}$ is a separating family of representations of $A$.
\end{enumerate}
\end{theorem}

In the proof, we use a trick with polar decomposition, which is folklore nowadays, but was first done in \cite{AP}.

\begin{proof}
If we assume (i), then for any $a\in A\backslash\{0\}$, we can choose $b\in A_\mathcal{F}\backslash\{0\}$ such that $\|a-b\|<\frac{1}{4}\|a\|$ and $\pi\in\mathcal{F}$ so that $\|\pi(b)\|=\|b\|$. Then, 
\begin{align*}
    \|a\|-\|\pi(a)\|=|\|a\|-\|b\|+\|\pi(b)\|-\|\pi(a)\||\leq \|a-b\|+\|\pi(b-a)\|<\frac{1}{2}\|a\|.
\end{align*}
Hence, $0<\frac{1}{2}\|a\|<\|\pi(a)\|$, i.e. $\mathcal{F}$ is a separating family of representations.

Now, assume (ii), and let $a\in A\backslash\{0\}$ and $\epsilon>0$. By Lemma \ref{sup} there exists  $\pi\in \mathcal{F}$ such that $\|a\|\leq \|\pi(a)\|+\epsilon$.
Embed $\tilde{A}$ into $B(\mathcal{H})$ for some $\mathcal{H}$, where $\tilde{A}$ is the unitization of $A$, and let
\begin{align*}
    a=u|a|
\end{align*}
be the polar decomposition of $a$ in $B(\mathcal{H})$. Define a function $f:\mathbb{R}^+\to \mathbb{R}^+$ by 
 \[
   f(t) = \left\{
     \begin{array}{lcl}
       t &;& t\in [0,\|\pi(a)\|]\\
       &&\\
       \|\pi(a)\| &;&  t\in(\|\pi(a)\|,\infty)
     \end{array}
   \right.
\]
 Let $
    b=uf(|a|).$
We claim that $b\in A_\mathcal{F}$ and $\|b-a\|<\epsilon$. First, note that $b\in A$. Indeed,  $f(t)=tg(t)$ where 
 \[
   g(t) = \left\{
     \begin{array}{lcl}
       1 &;&  t\in [0,\|\pi(a)\|]\\
       &&\\
        \displaystyle \frac{\|\pi(a)\|}{t} &;&  t\in(\|\pi(a)\|,\infty)
     \end{array}
   \right.
\]
Then $g$ is continuous on $[0,\infty)$, and $g(|a|)\in \tilde{A}$. Hence $b=uf(|a|)=ag(|a|)\in A$ since $A$ is an ideal in $\tilde{A}$.

To show that $b\in A_\mathcal{F}$, it will suffice to show that $\|b\|\leq \|\pi(b)\|$. If $A$ is non-unital,  let $\pi'$ denote the unique unital extension of $\pi$ to $\tilde{A}$, and if $A$ is unital, let $\pi' = \pi$.
Then, since $g(t)=1$ when $t\in [0,\|\pi(a)\|]$, we have that $\pi'(g(|a|))=g(\pi(|a|))=1$ in $\pi'(\tilde{A})$ and hence
\begin{align*}
    \pi(b)=\pi(ag(|a|))=\pi(a)g(\pi(|a|))=\pi(a).
\end{align*}
This gives us that
\begin{align*}
    \|b\| \leq  \|f(|a|)\|=\sup_{t\in\sigma(|a|)}|f(t)|\leq \|\pi(a)\|=\|\pi(b)\|,
\end{align*}
Finally, 
\begin{align*}
    \|a-b\|&=\|u|a|-uf(|a|)\|\\
    &\leq \||a|-f(|a|)\|\\
    &=\sup_{t\in \sigma(|a|)}|t-f(t)|\\
    &\leq \|a\|-\|\pi(a)\|\leq\epsilon.
\end{align*}
Hence, $A=\overline{A_\mathcal{F}}$.
\end{proof}

\begin{corollary}\label{RFDcor}
The following are equivalent for a $C^\ast$-algebra $A$:
\begin{enumerate}
    \item[(i)] The set $\displaystyle\{a\in A: \|a\|=\max_{\substack{\pi\in \text{Irr}_n(A)\\ n<\infty}}\|\pi(a)\|\}$
    is dense in $A$. 
    \item[(ii)] $A$ is RFD.
\end{enumerate}
\end{corollary}

A natural question now is how to characterize the class of  $C^\ast$-algebras for which every element attains its norm under some finite-dimensional representation.
For example, is this true for $C^\ast(\mathbb{F}_n)$?

It turns out that the answer is ``no" for any $C^\ast$-algebra that has an infinite-dimensional irreducible representation, including $C^*(\mathbb{F}_n)$. We will address this in the next section. 

\section{A Characterization of FDI $C^\ast$-algebras}\label{section4}

We begin with a key lemma that is intuitively clear and must be known to specialists.

\begin{lemma}\label{pointeval} Let  $T(B)$ be an AF-telescope with associated inductive sequence $(B_n)$. Then any  irreducible representation $(\pi,\mathcal{H})$ of $T(B)$ factorizes through a point evaluation $ev_t$, for some $t\in (0, \infty]$. Moreover, when $B_n$ are all simple and distinct, if $\dim\mathcal{H}\leq \dim B_n$ for some $n$, then $t\leq n$.
\end{lemma}

\begin{proof} Let $\pi$ be an irreducible representation of $T(B)$. Put \begin{equation}\label{I} I=\{f\in T(B)| f(\infty)=0\}.\end{equation} Note that $I$ is a closed ideal in $T(B)$ and so $\pi|_I$ is either irreducible or zero. If it is zero, then $\pi$ factorizes through $T(B)/I \simeq B$ and hence  through $ev_{\infty}$. So we assume now  that $\pi|_I$ is non-zero and irreducible. For each $n\geq 1$, define the closed ideal $I_n\triangleleft I$ by
    \begin{align*}
        I_n&:=\{f\in I| f(t)=0\ \forall\ t\geq n\}
           \end{align*}
Then $(I_n)$ is a nested sequence of closed, two-sided ideals with $I=\overline{\bigcup_n I_n}$. Thus there must exist $n$ such that $\pi|_{I_n}$ is non-zero and therefore irreducible. Let  \begin{align*}
        \tilde I_n&:= \{f\in T(B)|\ f(t)=f(n)\ \forall\ t\geq n\}.
   \end{align*}

Then $I_n$ is an ideal in $\tilde I_n$, and so $\pi|_{I_n}$ extends uniquely to an irreducible representation (in fact $\pi|_{\tilde I_n}$)  of $\tilde I_n$. So, it will be sufficient to prove that any irreducible representation, say $\rho$,  of $\tilde I_n$ factorizes through a point evaluation. 

We will prove it by induction. Clearly it holds for $\tilde{I}_1\simeq C_0(0,1]\otimes B_1$. Assume it holds for $(n-1)$. 
Let $J_n$ be the closed ideal in  $\tilde{I}_n$ defined by 
\begin{align*}
    J_n &= \{f\in \tilde I_n\;|\; f(t) =0 \; \text{for all}\; t\notin [n-1, n]\}\\
    &\simeq C_0(n-1, n)\otimes B_n.
\end{align*} 
If $\rho$ does not vanish on $J_n$, then it is irreducible on $J_n$ and hence factorizes through a point evaluation.
So we can assume  $\rho$ vanishes on $J_n$. Then $\rho$ factorizes through the map $\tilde I_n \to \tilde I_n/J_n \cong \tilde I_{n-1}$ given by the restriction $f\mapsto f|_{[0, n-1]}$ and hence $\rho$ factorizes through a point evaluation by the induction hypothesis.

Thus, $\pi|_{I_n}$ factorizes through a point evaluation.  
 Since an irreducible representation of an ideal extends uniquely to a representation of the whole $C^*$-algebra, we conclude that $\pi$ factorizes through a point evaluation. 
 
 Moreover, if each $B_n$ is simple, then any irreducible representation of $T(B)$ is equivalent to a point evaluation $ev_t$ for some $t\in (0,\infty]$, in which case the image of the representation is isomorphic to $B_{\lceil t\rceil}$.
\end{proof}

\begin{remark}
Recall that a $C^*$-algebra is {\it (n-)subhomogeneous} if all of its irreducible representations are of bounded finite dimension. Clearly any subhomogenous $C^*$-algebra is FDI, but there exist many FDI $C^*$-algebras that are not subhomogeneous. For instance, if $B$ is a UHF algebra or $K(\ell^2)$, then $I$ in (\ref{I}) is not subhomogeneous. 

More such examples come from group theory. In \cite{M}, Moore 
proves that a locally compact group has a finite bound for
the dimensions of its irreducible unitary representations if and only if it has an open
abelian subgroup of finite index. On the other hand, he also shows in \cite{M} that a locally compact group has all of its irreducible unitary representations of finite dimension if and only if it is a projective limit of Lie groups with the same property; and a Lie group has this property if and only if it has an open subgroup of finite index that is compact modulo its center. Consequently, examples of FDI but non-subhomogeneous $C^*$-algebras include, for instance, the full group $C^*$-algebra of a locally compact Lie group whose irreducible representations are all finite-dimensional but which has no open abelian subgroups of finite index. 

On the other hand, if $G$ is a discrete group, Thoma shows in \cite{T1} and \cite{T2} that all irreducible unitary representations of $G$ are finite-dimensional iff they are all of bounded finite dimension iff the group is type I iff the group is virtually abelian. In other words, for a discrete group $G$, the following are equivalent:
\begin{enumerate}
    \item $C^\ast(G)$ is subhomogeneous
    \item $C^*(G)$ is FDI
    \item $C^*(G)$ is GCR
    \item $G$ is virtually abelian.
\end{enumerate} 
\end{remark}

\begin{lemma}\label{example}
For any simple, infinite-dimensional AF-algebra $B$ with inductive sequence $(B_n)$, there is an element $f\in T(B)$ such that $\|\pi(f)\|<\|f\|=\|f(\infty)\|$ for any finite-dimensional representation $\pi$ of $T(B)$. 
\end{lemma}

\begin{proof}
Let $0\neq x\in B_1\subset B$ and define $f\in T(B)$ by $f(t)=(1-e^{-t})x$. Recall that any finite-dimensional representation $\pi$ of $T(B)$ is a finite direct sum of irreducible representations. Then, since $B$ has no finite-dimensional representations, by Lemma \ref{pointeval} there exists a finite set $F\subset(0,\infty)$ such that $\|\pi(f)\|=\max_{t\in F}\|f(t)\|$. In particular, since $\|f(t)\|$ is a strictly increasing function,  $\|\pi(f)\|<\|f(\infty)\|=\|f\|$. 
\end{proof}

Now we are ready to give the main theorem of this section. 

\begin{theorem}\label{FDI} The following are equivalent for any $C^\ast$-algebra $A$:
\begin{enumerate}
    \item[(i)] $A$ is FDI.
    \item[(ii)] For each $a\in A$ there exists a representation $(\pi,\mathcal{H})$ of $A$ with dim$(\mathcal{H})<\infty$ such that $\|a\|=\|\pi(a)\|$.
\item[(iii)] $A$ does not have  an infinite-dimensional simple AF-algebra as a subquotient.
\end{enumerate}
\end{theorem}
\begin{proof}
To see that (i) implies (ii), recall that for any $a\in A$, there exists a pure state $\varphi$ on $A$ such that $|\varphi(a^*a)|=\|a^*a\|$. Applying the GNS construction to $\varphi$ gives an irreducible representation $\pi_\varphi$ and unit vector $\xi_\varphi$ such that $$\|\pi_\varphi(a)\xi_\varphi\|=\|a\|.$$ Since $A$ is FDI, we know $\pi_\varphi$ is finite-dimensional. 

To show that (ii) implies (iii), suppose $A_0\subseteq A$ is a $C^*$-subalgebra, $B$ is a simple, infinite-dimensional AF-algebra with inductive sequence $(B_n)$, and $q:A_0\to B$ a surjective $\ast$-homomorphism. Let $T(B)$ be the mapping telescope for $(B_n)$. Since AF-telescopes are projective (\cite{LP}), there is a $\ast$-homomorphism
$\psi:  T(B) \to A_0$ such that $q\circ \psi = ev_\infty$, i.e. the following diagram commutes.
\[
\begin{tikzcd}
& A_0\arrow[two heads]{d}{q}\\
T(B)\arrow{ur}{\psi}\arrow{r}[swap]{ev_\infty} & B
\end{tikzcd}
\]
Let $f\in T(B)$ be the element guaranteed by Lemma \ref{example}, and let $a:=\psi(f)$. Then $\|a\|=\|f\|$ since
\begin{align*}
\|a\|=\|\psi(f)\|\leq \|f\|=\|f(\infty)\|=\|ev_\infty(f)\|=\|q(a)\|\leq \|a\|.
\end{align*}
If $\|a\|=\|\pi(a)\|$, for some finite-dimensional representation $\pi$ of 
$A$,  then $f$ attains its norm under the finite-dimensional representation $\pi\circ \psi$ of $T(B)$ which, by Lemma \ref{example}, is not true. Thus $\|a\| > \|\pi(a)\|$ for any finite-dimensional representation $\pi$ of $A$.

To show that (iii) implies (i),  we notice first that (iii) implies that $A$ is GCR. Indeed otherwise  $A$ would have a subquotient isomorphic to the CAR algebra $\mathbb{M}_{2^\infty}$ by Theorem \ref{non-GCR}. Assume now that $A$ does have an infinite-dimensional irreducible representation  $(\pi,\mathcal{H})$. Since $A$ is GCR,  $K(\mathcal{H})\subseteq \pi(A)$. Let $\mathcal{H}'\subseteq \mathcal{H}$ be an infinite-dimensional separable subspace, and let $P_{\mathcal{H}'}$ denote the projection of $\mathcal{H}$ onto $\mathcal{H}'$. Since 
$K(\mathcal{H}')\oplus 0|_{\mathcal{H}'^\perp}$
is singly generated, we can choose  $x\in A$ such that $\pi(C^*(x))=K(\mathcal{H}')\oplus 0|_{\mathcal{H}'^\perp}$. Then $C^*(x)$ is a subalgebra of $A$, and $P_{\mathcal{H}'}\pi P_{\mathcal{H}'}: C^*(x)\to K(\mathcal{H}') $ is  a  surjective $\ast$-homomorphism .
\end{proof}

\begin{remark}
Rephrasing the theorem, we can say that a $C^\ast$-algebra $A$ contains an element $a$ with $\|a\|>\|\pi(a)\|$ for any finite-dimensional representation $(\pi,\mathcal{H})$ of $A$ if and only if $A$ has an infinite-dimensional irreducible representation. 
  Since $C^*(\mathbb{F}_n)$ is primitive (\cite{C}), we conclude that there are elements which do not attain their norm under a finite-dimensional representation.  Recall that in \cite{FNT}, the authors show that no such element lies in $\mathbb{CF}_n$.
\end{remark}

\begin{remark}\label{subh}
It follows from Theorem \ref{RFD} and standard arguments that the following are equivalent for a $C^\ast$-algebra $A$ and any $n<\infty$.
\begin{enumerate}
\item[(i)] $A$ is $n$-subhomogeneous (i.e. every irreducible representation is of dimension no more than $n$). 
\item[(ii)] $A$ has a separating family of finite-dimensional representations of dimension no more than $n$. 
\item[(iii)] For each $a\in A$ there exists a representation $(\pi,\mathcal{H})$ of $A$ of dimension no more than $n$ such that $\|a\|=\|\pi(a)\|$.
\item[(iv)] The set $\displaystyle \{a\in A: \|a\|=\max_{\substack{\pi\in \text{Irr}_k(A)\\ k\leq n}}\|\pi(a)\|\}$
is dense in $A$. 
\end{enumerate}
\end{remark}

Before we conclude this section, we record a consequence of this remark, which will prove useful in the next section. 
\begin{proposition}\label{subhom}
Suppose $A$ is RFD and $A_0\subseteq A$ is a non-subhomogeneous subalgebra. Then there exists an unbounded sequence $(n_k)_{k\in\mathbb{N}}$ in $\mathbb{N}$ and irreducible representations $\pi_k:A_0\to \mathbb{M}_{n_k}$ such that each $\pi_k$ is a subrepresentation of $\pi'_k|_{A_0}$, denoted $\pi_k\leq \pi'_k|_{A_0}$, for some finite-dimensional representation $\pi_k$ of $A$.
\end{proposition}

\begin{proof}
Let $\mathcal{F}$ be a separating family of finite-dimensional irreducible representations of $A$. Then the collection $\{\pi|_{A_0}: \pi\in \mathcal{F}\}$ is a separating family of representations of $A_0$. Let
$$\mathcal{F}_0=\{\sigma\in \text{Irr}(A_0): \sigma \leq \pi|_{A_0}\ \text{for some}\ \pi\in \mathcal{F}\}.$$
Then $\mathcal{F}_0$ separates the points of $A_0$. If the set $\{\text{dim}(\sigma)|\ \sigma\in \mathcal{F}_0\}$ is bounded, then $A_0$ is subhomogeneous by Remark \ref{subh}. 
\end{proof}

\section{Growth of Finite-Dimensional Norms}\label{section5}
Let $n\in\mathbb{N}$. If a $C^*$-algebra has a representation of dimension no more than $n$, we define a seminorm $\|\cdot\|_{\mathbb{M}_n}$ on $A$ by 
\begin{align*}
    \|a\|_{\mathbb{M}_n}=\sup\{\|\pi(a)\|\;|\; \pi:A\to \mathbb{M}_n\},
\end{align*}
for all $a\in A$. We do not require  representations to be non-degenerate and so by $\pi:A \to \mathbb{M}_n$ we mean a representation of dimension not larger than $n$. Equivalently we can say that 
\begin{align*}
    \|a\|_{\mathbb{M}_n}=\sup\{\|\pi(a)\|\}
\end{align*}
 where supremum is taken over all irreducible representations of dimension not larger than $n$.

Suppose  that $\{n_1, n_2, \ldots\}$ is the nonempty set of dimensions of all \textit{irreducible} finite-dimensional representations of a $C^*$-algebra $A$, arranged in increasing order, with $\kappa=|\{n_1, n_2,...\}|$.
Then for each $a\in A$ we get a sequence $$(\|a\|_{\mathbb{M}_{n_k}})_{k\leq \kappa}.$$ In general we would like to know what sequences of numbers can be obtained in this way.
Namely, define the set $\Lambda(A)$ by
    $$\Lambda(A)=\{(\|a\|_{\mathbb{M}_{n_k}})_{k\leq \kappa}\;|\: a\in A\}$$
where in the case $\kappa=\aleph_0$ by $(\lambda_k)_{k\leq \kappa}$ we mean $(\lambda_k)_{k\in \mathbb N}$.
Since we allow degenerate representations, all such sequences will be nondecreasing.
In this section we prove that, for any $C^*$-algebra $A$ with at least one finite-dimensional representation, $\Lambda(A)$ contains the set of all nondecreasing sequences of $\kappa$ positive numbers which are eventually constant (Theorem \ref{1}). 
 Moreover, we show that $A$ is an FDI-algebra if and only if the two sets coincide. (Corollary \ref{3}).
\medskip

Below we will use the Chinese Remainder Theorem: Let $A$ be a $C^*$-algebra and $I_1, \ldots, I_k$ be closed two-sided ideals in $A$ such that $I_i + I_j = A$, when $i\neq j$. Then the map $$\phi: a+ \bigcap_{i=1}^k I_i \mapsto (a+I_1, \ldots,  a+I_k)$$ gives a $\ast$-isomorphism from $A/(\bigcap_{i=1}^k I_i)$ to $A/I_1 \oplus \ldots \oplus  A/I_k.$

\begin{theorem}\label{1} Let $N\in \mathbb N$, and $0\le \lambda_1 \le \lambda_2 \le ... \le \lambda_N$ be a sequence of nonnegative numbers. Suppose that a $C^*$-algebra $A$ has irreducible representations of dimensions $n_1< n_2< \ldots < n_N$ (and possibly of some other dimensions too). Then there exists $a\in A$ such that $\|a\|_{\mathbb{M}_{n_k}} = \lambda_k$, for $1\le k \le N.$
In addition $a$ can be chosen such that $\|a\| = \|a\|_{\mathbb{M}_{n_N}}$.
\end{theorem}

\begin{proof} For each $i\leq N$, let $\pi_i:A\to \mathbb{M}_{n_i}$ be an irreducible representation with kernel $I_i$, i.e. $A/I_i\simeq \pi_i(A)=\mathbb{M}_{n_i}$. Since each $\mathbb{M}_{n_i}$ is simple, each $I_i$ is a maximal ideal, and so
 $I_i + I_j = A$, for each $i\neq j$. Since $\ker (\oplus_{i=1}^N \pi_{i}) = 
\bigcap_{i=1}^N I_i$, we have $\left(\oplus_{i=1}^N \pi_{i}\right) (A) \simeq A/(\bigcap_{i=1}^N I_i)$ and hence by the Chinese Reminder Theorem $$\left(\oplus_{i=1}^N \pi_{i}\right) (A) = \pi_{1} (A) \oplus \ldots \oplus \pi_{N} (A).$$ Thus 
$q=\oplus_{i=1}^N \pi_{i}: A \to \mathbb{M}_{n_1} \oplus \ldots \oplus \mathbb{M}_{n_N}$ is a surjective $\ast$-homomorphism. Now, consider standard embeddings $$\mathbb{M}_{n_1}\subset  \mathbb{M}_{n_2} \subset  \ldots \subset \mathbb{M}_{n_N},$$ and let $T(\mathbb{M}_{n_N})$ denote the corresponding AF-telescope. For $i\leq N$, let \\ 
$ev_i:T(\mathbb{M}_{n_N})\to \mathbb{M}_{n_i}$ denote the evaluation map, and
 let $$\phi = \oplus_{i=1}^N ev_{i}: T(\mathbb{M}_{n_N}) \to  \mathbb{M}_{n_1} \oplus \ldots \oplus \mathbb{M}_{n_N}.$$ Since $T(\mathbb{M}_{n_N})$ is projective, $\phi$ lifts to some $*$-homomorphism $\psi: T(\mathbb{M}_{n_N}) \to A$ so that $$q\circ \psi = \phi,$$ giving us the following commutative diagram.
\[
\begin{tikzcd}
& A\arrow[two heads]{d}{q}\\
T(\mathbb{M}_{n_N})\arrow{ur}{\psi}\arrow{r}[swap]{\phi} & \oplus_{i=1}^N \mathbb{M}_{n_i}
\end{tikzcd}
\]
\medskip

Let $f \in T(\mathbb{M}_{n_N})$ be any element such that $\|f(t)\|$ is a nondecreasing function on $(0,\infty]$ with $\|f(i)\| = \lambda_i$ for each $i\leq N$, and let $a= \psi(f)$. Then $$q(a)=\oplus_{i=1}^N \pi_i(a) = \oplus_{i=1}^N \pi_i(\psi(f)) = \phi(f) = \oplus_{i=1}^N f(i).$$ Hence for any $k\leq N$, $$\|\pi_k(a)\| = \|f(k)\| = \lambda_k,$$
which implies $$\|a\|_{\mathbb{M}_{n_k}} \ge \lambda_k.$$ 

On the other hand, for any representation $\pi$ of $A$ of dimension not larger than $n_k$, $\pi\circ \psi$ is a representation  of $T(\mathbb{M}_{n_N})$ of dimension not larger than $n_k$ and hence factors through a finite direct sum of evaluations at some points in $(0,k]$ by Lemma \ref{pointeval}. 
Since $\|f(t)\|$ is a nondecreasing function and $\|f(k)\| = \lambda_k$, it follows that $$\|\pi(a)\| = \|\pi(\psi(f))\| \le \lambda_k.$$ 
Thus for each $k\leq N$, $$ \|a\|_{\mathbb{M}_{n_k}} = \lambda_k.$$ If we additionally chose $f$ to attain its norm at the point $n_N$, then we would have $$\lambda_N  = \|f\| \ge \|\psi(f)\| = \|a\| \ge \|q(a)\| = \|\oplus_{k=1}^N f(k)\| = \lambda_N$$ whence $\|a\| = \lambda_N.$
\end{proof}

\begin{corollary}\label{Thom+}
Let $G$ be a discrete group with representations of dimensions\\ $n_1<n_2<\infty$, and let $\epsilon>0$. Then there is an $a\in \mathbb{C}G$ such that $\|a\|_{\mathbb{M}_{n_1}} \le \epsilon$ and $\|a\|_{\mathbb{M}_{n_2}} > \|a\| - \epsilon.$
\end{corollary}
\begin{proof}
Since $\mathbb{C} G$ is dense in $C^*(G)$, the statement follows from Theorem \ref{1}.
\end{proof}

\begin{remark}
 In case $G=\mathbb{F}_n$ with $n<\infty$, Thom proves in \cite{Thom} that
for any $d>0$ and $\epsilon>0$, there exists a nontrivial $w\in \mathbb{F}_n$ such that 
$$\sup\{\|\pi(1-w)\|: \pi:C^*(\mathbb{F}_n)\to \mathbb{M}_d\}<\epsilon.$$
 However, he does not state (in \cite{Thom}) for which $m>d$ we know that we have $\|1-w\|_{\mathbb{M}_m} > \|1-w\| - \epsilon$. Corollary \ref{Thom+} guarantees that for any $m>d$, we can find some element with this behavior, but, in contrast to Thom, we do not know what this element looks like.
 
 On the other hand, it follows from \cite[Lemma 2.7]{FNT} that it suffices to take $m\geq 4n^\ell$, where $\ell$ is the length of $w$. We show in Section \ref{section6} that it actually suffices to take $m\geq 2\ell$.
\end{remark}

\begin{corollary}\label{3} 
Assume $A$ has irreducible representations of $\kappa>0$ many distinct finite dimensions.  
Then
\begin{align*}
\Lambda(A) \supseteq \{(\lambda_n)_{n\leq \kappa} \;|\; 0\le \lambda_n \le \lambda_{n+1}\ \forall\ n\leq \kappa\ \text{and}\ (\lambda_n)\ \text{is eventually constant}\}
\end{align*}
 Moreover, the two sets are equal if and only if $A$ is FDI.

\end{corollary}

\begin{proof} It follows from Theorem \ref{1} and Theorem \ref{FDI}.
\end{proof}

\medskip

Using the techniques from this section, we can provide further characterizations for FDI $C^*$-algebras. 

 \begin{theorem}\label{subspace}
Suppose $A$ is RFD. Then the following are equivalent:
\begin{enumerate}
    \item A is FDI.
    \item The set 
$$\{a\in A: \|a\|=\max_{\substack{\pi\in \text{Irr}_n(A)\\ n<\infty}}\|\pi(a)\|\}$$
is closed under addition.
\item The set 
$$\{a\in A: \|a\|=\max_{\substack{\pi\in \text{Irr}_n(A)\\ n<\infty}}\|\pi(a)\|\}$$
is closed under multiplication.
\end{enumerate}
\end{theorem}
 
 \begin{proof}
 Clearly $(1)\Rightarrow (2)$ and $(1)\Rightarrow (3)$.
 
 To show that $(2)\Rightarrow (1)$, we assume $A$ is not FDI. We demonstrate the existence of $a_1,a_2\in A$ such that $a_1$ and $a_2$ achieve their norm under a finite-dimensional representation, but $a_1+a_2$ does not. 
 By Theorem \ref{FDI}, there exists a subalgebra $A_0\subseteq A$ and a simple infinite-dimensional AF-algebra $B$ with inductive sequence $(B'_n)$ such that $A_0$ surjects onto $B$. 
 Moreover, we can take $B$ to be either $\mathbb{M}_{2^\infty}$ with inductive sequence $(\mathbb{M}_{2^n})$ or $\mathbb{K}(\ell^2)$ with inductive sequence $(\mathbb{M}_n)$. 
 By Proposition \ref{subhom}, there is a finite-dimensional nonzero irreducible representation $\pi_0$ of $A_0$ and a finite-dimensional representation $\pi_0'$ of $A$ such that $\pi_0$ is a subrepresentation of $\pi_0'|_{A_0}$. Then for some $n$, $$B'_{n-1}\hookrightarrow \pi_0(A_0)\hookrightarrow B'_n.$$ 
 So, we can find a new inductive sequence $(B_k)$ where $B_k=B'_k$ for $k< n$, $B_n\simeq \pi_0(A_0)$, and $B_k=B'_{k-1}$ for $k> n$. 
 We still call the inductive limit $B$, and let $\pi:A_0\to B$ be a surjection. As in the proof of Theorem \ref{1}, since $B_n$ and $B$ are simple, it follows from the Chinese remainder theorem that $$q:=\pi\oplus \pi_0:A_0 \to B\oplus B_n$$ is a surjection. Let $$\phi:=ev_\infty\oplus ev_n: T(B)\to B\oplus B_n.$$ The projectivity of $T(B)$ again yields the following commutative diagram.
 \[
\begin{tikzcd}
& A_0\arrow[two heads]{d}{q}\\
T(B)\arrow{ur}{\psi}\arrow{r}[swap]{\phi} & B\oplus B_n
\end{tikzcd}
\]
Let $b\in B_1$ with $\|b\|=1$. Define $f_1\in T(B)$ by 
 \[
   f_1(t) = \left\{
     \begin{array}{lcl}
       \frac{2}{n}tb &;&  t\in (0,n]\\
       &&\\
       \frac{2}{n}(2n-t)b &;&  t\in(n,2n]\\
       &&\\
       (1-e^{2n-t})b &;& t\in(2n,\infty].
     \end{array}
   \right.
\]
Define $f_2\in T(B)$ by 
 \[
   f_2(t) = \left\{
     \begin{array}{lcl}
       -f_1(t) &;&  t\in (0,2n]\\
       &&\\
       0 &;&  t\in(2n,\infty].
     \end{array}
   \right.
\]
Then for $i=1,2$, $$\|f_i\|=\|f_i(n)\|,$$ but for any finite-dimensional representation $\rho$ of $T(B)$, $$\|f_1+f_2\|=\|(f_1+f_2)(\infty)\|=1>\|\rho(f_1+f_2)\|.$$ Let $a_i=\psi(f_i)$ for $i=1,2$. By the same argument as in Theorem \ref{FDI}, $\|a_i\|=\|f_i\|$ for $i=1,2$, and $\|a_1+a_2\|=\|f_1+f_2\|$. Also as in the proof of Theorem \ref{FDI}, $\|a_1+a_2\|>\|\rho(a_1+a_2)\|$ for any finite-dimensional representation $\rho$ of $A$. 
Since $\|a_i\|=\|f_i\|$ for $i=1,2$, we know that each $a_i$ attains its norm under some finite-dimensional representation of $A_0$. To see that they will attain their norms under some finite-dimensional representation of $A$, note that since
$$q(a_i)=f_i(\infty)\oplus f_i(n)$$ for $i=1,2$, we have that $$\|\pi_0(a_i)\|=\|f_i(n)\|=\|a_i\|$$ for $i=1,2$. Because $\pi_0\leq \pi'_0|_{A_0}$, it follows that $\|\pi_0(a_i)\|=\|a_i\|$ for $i=1,2$.

To show that $(3)\Rightarrow (1)$, assume that $A$  is not FDI. Again, by Theorem \ref{FDI}, there exists a subalgebra $A_0\subseteq A$ and a simple infinite-dimensional AF-algebra $B$ with inductive sequence $(B'_n)$ such that $A_0$ surjects onto $B$, and again, we take $B$ to be either $\mathbb{M}_{2^\infty}$ or $\mathbb{K}(\ell^2)$. 
By the Proposition \ref{subhom}, there exist irreducible representations $\pi_1$ and $\pi_2$ of $A_0$ and $\pi'_1$ and $\pi'_2$ of $A$ such that $\pi_i:A_0\to \mathbb{M}_{k_i}$ with $k_1<k_2$ and $\pi_i\leq \pi'_i|_{A_0}$ for $i=1,2$. 
As before, we intertwine $\mathbb{M}_{k_1}$ and $\mathbb{M}_{k_2}$ into the inductive sequence $(B_n)$ of $B$ and let $n_1<n_2$ so that now $B_{n_i}=\mathbb{M}_{k_i}$ for $i=1,2$. Let $\pi:A_0\to B$ be a surjection. As before, we use the Chinese remainder theorem to get a surjection 
$$q:=\pi\oplus \pi_1\oplus \pi_2: A_0\to B\oplus B_{n_1}\oplus B_{n_2}.$$
Letting $$\phi:=ev_\infty\oplus ev_{n_1}\oplus ev_{n_2}: T(B)\to B\oplus B_{n_1}\oplus B_{n_2},$$ the projectivity of $T(B)$ yields the following commutative diagram.  
\[
\begin{tikzcd}
& A_0\arrow[two heads]{d}{q}\\
T(B)\arrow{ur}{\psi}\arrow{r}[swap]{\phi} & B\oplus B_{n_1}\oplus B_{n_2}
\end{tikzcd}
\]
Let $b\in B_1$ self-adjoint with $\|b\|=1$. Define $f_1\in T(B)$ by 
 \[
   f_1(t) = \left\{
     \begin{array}{lcl}
       \frac{t}{n_1}b &;&  t\in (0,n_1]\\
       &&\\
       \frac{-1}{n_2-n_1}(t-n_2)b &;&  t\in(n_1,n_2]\\
       &&\\
       (1-e^{n_2-t})b &;& t\in(n_2,\infty].
     \end{array}
   \right.
\]

Define $f_2\in T(B)$ by 
 \[
   f_2(t) = \left\{
     \begin{array}{lcl}
       \frac{t}{n_2}b &;&  t\in (0,n_2]\\
       &&\\
       b &;& t\in(n_2,\infty].
     \end{array}
   \right.
\]

Then, $\|f_i\|=\|f_i(n_i)\|$ for $i=1,2$, and $$\|f_1f_2\|=\|f_1f_2(\infty)\|=\|b^2\|=1>\|f_1f_2(t)\|$$ for all $t<\infty$.
Let $a_1=\psi(f_1)$ and $a_2=\psi(f_2)$. As before, $\|a_i\|=\|f_i\|$ and $\|a_1a_2\|=\|f_1f_2\|$. Moreover, $\|a_1a_2\|>\|\rho(a_1a_2)\|$ for all finite-dimensional representations $\rho$ of $A$. Also as before, we have for $i=1,2$, $$\|a_i\|\geq \|\pi'_i(a_i)\|\geq \|\pi_i(a_i)\|=\|f_i(n_i)\|=\|a_i\|.\qedhere$$
 \end{proof}
 
 As a consequence, if a $C^*$-algebra is RFD but not FDI, then the set of elements in a $C^*$-algebra that achieve their norm under a finite-dimensional representation does not form an algebra. In particular, we have the following corollary.
 
 \begin{corollary}
 For $n<\infty$, there exists an element in $C^*(\mathbb{F}_n)\backslash \mathbb{CF}_n$ that achieves its norm under some finite-dimensional representation.
 \end{corollary}
 

Theorem \ref{1} says we can find an element that attains the prescribed norms $\lambda_k$ for $k\leq N$. Theorem \ref{FDI} says that if a $C^*$-algebra is not FDI, then we can find an element that does not achieve its norm under any finite-dimensional representation. The following theorem says that if we assume the $C^*$-algebra is RFD and not FDI, then we can find an element that does both.

\begin{theorem}\label{RFDnotFDI} Suppose $A$ is RFD and has an infinite-dimensional irreducible representation. Then there exist $M_1 < M_2 < \ldots<\infty $ such that for any finite sequence
$0\le \lambda_1 \le \lambda_2 \le ... \le \lambda_N \le \lambda$ there is $a\in A$ such that $\|a\|_{\mathbb{M}_{M_k}} = \lambda_k$ for $1\le k \le N$, $\|a\| = \lambda$, and, in the case $\lambda > \lambda_N$, $\|a\|\neq \|a\|_{\mathbb{M}_{M_k}}$ for any $k<\infty$.   
\end{theorem}

\begin{proof} 
Since $A$ has an infinite-dimensional irreducible representation, by Theorem \ref{FDI} there is a $C^*$-subalgebra $A_0\subseteq A$, a simple infinite-dimensional AF-algbera $B$, and a surjective $\ast$-homomorphism $\pi: A_0 \to B$. Again, we can and do take $B$ to be either $\mathbb{M}_{2^\infty}$ or $\mathbb{K}(\ell^2)$. 

Since $A_0$ is a $C^*$-subalgebra of an RFD $C^*$-algebra, it is also RFD. Since $B$ is a quotient of $A_0$, $A_0$ is not subhomogeneous. It then follows from Proposition \ref{subhom} that we can build sequences $(j_i), (m_i)$, and $(M_i)$ of positive integers such that
$$j_1<2^{j_1} < m_1 \leq M_1<j_2<2^{j_2} < m_2 \leq M_2< \ldots, $$
and such that for each $i$, there exists an irreducible representation $\pi_i$ of $A_0$ and a representation $\pi'_i$ of $A$ such that dim$(\pi_i)=m_i$, dim$(\pi'_i)=M_i$, and $\pi_i\leq \pi'_i|_{A_0}$.

Since  $B$ is simple, using the Chinese Reminder Theorem in the same way as in the proof of Theorem \ref{1}, we conclude that 
   the $\ast$-homomorphism $$q=(\oplus_{i=1}^N\pi_{i})\oplus \pi: A_0 \to \mathbb{M}_{m_1}\oplus \ldots \oplus \mathbb{M}_{m_N} \oplus B$$ is surjective. 
   If $B = \mathbb{K}(\ell^2)$ let 
   $$\phi = (\oplus_{i=1}^N  (ev_{m_i})) \oplus ev_{\infty}: T(B) \to  \mathbb{M}_{m_1} \oplus \ldots \oplus \mathbb{M}_{m_N} \oplus B.$$  If $B = \mathbb{M}_{2^{\infty}}$ let  $$\phi = (\oplus_{i=1}^N  (ev_{j_i})) \oplus ev_{\infty}: T(B) \to  \mathbb{M}_{m_1} \oplus \ldots \oplus \mathbb{M}_{m_N} \oplus B$$ where we view $\mathbb{M}_{2^{j_i}}$ as a subalgebra of $\mathbb{M}_{m_i}$ via the standard inclusion.
   Since $T(B)$ is projective, $\phi$ lifts to some $*$-homomorphism $\psi: T(B) \to A_0$. Thus $$q\circ \psi = \phi.$$  

   Choose $f\in C_0(0,\infty]$ so that $f$ is nonnegative, nondecreasing, $f(\infty)=\lambda$, and for each $i\leq N$,
   $$f|_{[j_i,M_i]}\equiv \lambda_i;$$ in case $\lambda>\lambda_N$, we also require $f(t)<\lambda$ for each $t<\infty$. We can view $f$ as an element of $T(B)$ by the standard embdedding of $\mathbb{C}$ to $B$. Let $a=\psi(f)$.

   By Lemma \ref{pointeval}, for any $i\in \mathbb{N}$ and any representation $\rho$ of $A$ of dimension not larger than $M_i$  we have 
   $$\|\rho(a)\| = \|\rho(\psi(f))\| \le \|f(M_i)\|.$$ Hence for any $i\le N$
   \begin{equation}\label{le}\|a\|_{\mathbb{M}_{M_i}} \le \lambda_i,\end{equation} and for all $i\in \mathbb{N}$, \begin{equation}\label{lelambda}\|a\|_{\mathbb{M}_{M_i}} \le \lambda.\end{equation}
   On the other hand, when $B= \mathbb{K}(\ell^2)$, 
   $$(\pi_{1}(a), \ldots, \pi_{N}(a), \pi(a)) = q(a) = q(\psi(f)) = \phi(f) = (f(m_1), \ldots, f(m_N), f(\infty)),$$ 
   and we conclude that $\pi_{i}(a) = f(m_i)=\lambda_i$ for $i\le N$; when $B = \mathbb{M}_{2^{\infty}}$, 
   $$(\pi_{1}(a), \ldots, \pi_{N}(a), \pi(a)) = q(a) = q(\psi(f)) = \phi(f) = (f(j_1), \ldots, f(j_N), f(\infty)),$$ 
  and we conclude that $\pi_i(a) = f(j_i)=\lambda_i$ for $i\le N$.
      In either case,
   \begin{equation}\label{ge}\|a\|_{\mathbb{M}_{M_i}} \ge \|\pi_i'(a)\|\ge \|\pi_i(a)\|=  \lambda_i,\end{equation} for $i\le N$.
   By (\ref{le}) and (\ref{ge}) we have for each $i\leq N$, \begin{equation}\label{equallambdai}\|a\|_{\mathbb{M}_{M_i}} = \lambda_i.\end{equation}
   We also have $$\lambda = \|\phi(f)\|  = \|q(\psi(f))\| =\|q(a)\| \le \|a\| = \|\psi(f)\| \le \|f\| = \lambda,$$ and thus \begin{equation}\label{normlambda}\|a\| = \lambda.\end{equation} By (\ref{lelambda}), (\ref{equallambdai}) and (\ref{normlambda}) we are done.   
  \end{proof}

\section{A Bound on Dimension}\label{section6}

Fritz, Netzer, and Thom show in \cite[Lemma 2.7]{FNT} that, for any $n<\infty$, any element in $\mathbb{CF}_n$ will achieve its norm in a representation of dimension no more than $4n^\ell$, where $\ell$ is the length of the longest word in the support of the element. In this section, we improve the bound on the dimension for binomials in $\mathbb{CF}_n$. Namely, we show the following proposition.

\begin{proposition}\label{dimestimate}
Let $\alpha,\beta\in \mathbb{C}$. Let $w_1,w_2\in\mathbb{F}_n$ be distinct reduced words for some $n<\infty$, and let $\ell$ be the length of the reduced word $w_2^{-1}w_1$. Then, there exists a representation $\pi:C^*(\mathbb{F}_n)\to \mathbb{M}_{2\ell}$ such that 
\begin{align*}
    \|\pi(\alpha w_1+\beta w_2)\|=|\alpha|+|\beta|.
\end{align*}
\end{proposition}

We say a word $w$ in $\mathbb{F}_n$ is \textit{balanced} if any representation $C^*(\mathbb{F}_n)\to \mathbb{C}$ maps $w\mapsto 1$, e.g. $w=x_1x_2x_1^{-1}x_2^{-1}$.
 Notice that if the reduced word $w_2^{-1}w_1$ is not balanced, then this norm will be achieved by a representation $C^\ast(\mathbb{F}_n)\to \mathbb{C}$ sending all but one generator to $1$. 
For example, the word $w=x_1x_2x_1^{-2}$ is not balanced, and the map $C^*(\mathbb{F}_2)\to\mathbb{C}$ sending $x_1\to 1$ and $x_2\to -1$ will send $1-w$ to an element in $\mathbb{C}$ of norm 2.  However, for a balanced word there is no such map. Since balanced words constitute the class of nontrivial examples, we will assume $w:=w_2^{-1}w_1$ is balanced.

Notice that for any representation $\pi:C^*(\mathbb{F}_n)\to B(\mathcal{H})$, 
\begin{align*}
\|\pi(\alpha w_1+\beta w_2)\|&=\|\alpha \pi(w_2^{-1} w_1)+\beta I_\mathcal{H}\|\\
 &=\|\alpha \pi(w)+\beta I_\mathcal{H}\|\\
 &=\max_{\lambda\in \sigma(\pi(w))}|\alpha\lambda+\beta|\\
\end{align*}
where $ w:= w_2^{-1} w_1$ and $\sigma(\pi(w))$ denotes the spectrum of an operator $\pi(w)\in B(\mathcal{H})$. This maximum equals $|\alpha|+|\beta|$ if and only if $\frac{\text{sgn}(\beta)}{\text{sgn}(\alpha)}\in \sigma (\pi(w))$, and so Proposition \ref{dimestimate} will follow from the following theorem. 

\begin{theorem}\label{eigenvalue}
For any nontrivial, balanced, reduced word $ w\in \mathbb{F}_n$ of length $\ell$ and any $\lambda\in \mathbb{T}$, there exists a representation $\pi:C^*(\mathbb{F}_n)\to \mathbb{M}_{2\ell}$ such that $\lambda\in \sigma(\pi(w))$. 
\end{theorem}
 For the proof we will first need two lemmas. 

  \begin{lemma}\label{Sigma_d}
  Let $ w\in\mathbb{F}_n$ be a nontrivial, balanced, reduced word. Define for each $1\leq d<\infty$
 \begin{align*}
     \Sigma_d:=\bigcup \{\sigma(\pi( w))|\ \pi:C^*(\mathbb{F}_n)\to \mathbb{M}_d\}.
 \end{align*}
 Then for each $1\leq d<\infty$, there is a $\theta_d\in [0,\pi]$ such that $$\Sigma_d=\{e^{i\theta}|\ \theta\in [-\theta_d,\theta_d]\}$$ with $\theta_d\leq \theta_{d+1}$ for each $1\leq d<\infty$.
 \end{lemma}
 
 \begin{proof}
 Since $(\Sigma_d)$ is clearly a nested sequence, we need only to show that for each $1\leq d<\infty$, 
 $$\Sigma_d=\{e^{i\theta}|\ \theta\in [-\theta_d,\theta_d]\}$$ for some $\theta_d\in [0,\pi]$. 
 
 Fix $1\leq d<\infty$. We first, observe that $\Sigma_d$ is symmetric about $\mathbb{R}$ and centered at $1$. Indeed, since $\lambda\in \sigma( w(u_1,...,u_n))$ for some $u_1,...,u_n\in\mathcal{U}(d)$ implies that $\overline{\lambda}\in\sigma( w(\overline{u_1},...,\overline{u_n}))$, where $\overline{u}$ denotes the complex conjugate of the matrix $u$. Clearly $1\in \Sigma_d$. 
 
By these observations, it will suffice to show that $\Sigma_d$ is the union of two continuous images of the compact, path-connected space $\prod_{k=1}^n \mathcal{U}(d)$, which have nontrivial intersection. 

To that end, let $w:\prod_{k=1}^n \mathcal{U}(d)\to \mathcal{U}(d)$ denote the word map; let $\psi:\mathcal{U}(d)\to [-1,1]$ be given by
$u\mapsto \min_{\lambda\in \sigma(u)}\text{Re}(\lambda),$
and let $\gamma:[-1,1]\to \{e^{i\theta}|\ \theta\in[0,\pi]\}$ be given by 
$r\mapsto r+i\sqrt{1-r^2}.$
Then, $\gamma\circ\psi\circ w$ is a continuous map on $\prod_{k=1}^n \mathcal{U}(d)$ whose image is the compact, path-connected arc $\{e^{i\theta}|\ \theta\in[0,\theta_d]\}$ where\\ $\theta_d=\max\{\theta\in [0,\pi]|\ e^{i\theta}\in\Sigma_d\}$. 
Likewise, the image of $\bar{\gamma}\circ \psi\circ w$ is the path-connected arc $\{e^{i\theta}|\ \theta\in[-\theta_d,0]\}$ for the same $\theta_d\in [0,\pi]$.
Hence, the union of the two is the compact, path connected arc $\{e^{i\theta}|\ \theta\in [-\theta_d,\theta_d]\}=\Sigma_d.$
   \end{proof}

By Lemma \ref{Sigma_d}, we can conclude that, if for some $1\leq d<\infty$ there exists a representation $\pi:C^*(\mathbb{F}_n)\to \mathbb{M}_d$ such that $-1\in \sigma(\pi(1-w))$, then $\Sigma_d=\mathbb{T}.$
The following lemma tells us for which $1\leq d<\infty$ we can expect such a representation. 
\begin{lemma}\label{-1}
For any nontrivial, balanced, reduced word $ w\in \mathbb{F}_n$ of length $\ell$, there are (permutation) matrices $u_1,...,u_n\in \mathcal{U}(2\ell)$ such that $-1\in\sigma( w(u_1,...,u_n))$. 
\end{lemma}

\begin{proof}
The goal will be to construct a map from $\mathbb{F}_n$ into $\mathcal{S}_{2\ell}$, which maps $ w$ to a permutation in $\mathcal{S}_{2\ell}$ containing the two-cycle $(1,\ell+1)$. Composing this map with the natural inclusion of $\mathcal{S}_{2\ell}$ into $\mathcal{U}(2\ell)$ will yield a map from $\mathbb{F}_n$ to $\mathcal{U}(2\ell)$ mapping $ w$ to a permutation matrix with $-1$ as an eigenvalue. 

Write $ w=x_{i_\ell}^{\epsilon_\ell}\cdot\cdot\cdot x_{i_1}^{\epsilon_1}$ where $\epsilon_i\in\{\pm 1\}$ and $i_k\in \{1,...,\ell\}$. 

First, we claim that we may assume $i_\ell\neq i_1$. 
Indeed, suppose $i_\ell=i_1$. Let $d\geq 0$ and $u_1,...,u_n\in \mathcal{U}(d)$, and denote $w(u_1,...,u_n):=u_{i_\ell}^{\epsilon_\ell}\cdot\cdot\cdot u_{i_1}^{\epsilon_1}$. 
If $\epsilon_1=-\epsilon_\ell$, then $$\sigma(w(u_1,...,u_n))=\sigma(u_{i_1}^{-\epsilon_1} w(u_1,...,u_n) u_{i_1}^{-\epsilon_1}).$$ 
This reduction will not change the length of the word and must terminate since $ w$ is balanced, reduced, and nontrivial. Hence, we may assume $\epsilon_1=\epsilon_\ell$. Moreover, since $-1\in \sigma(w(u_1,...,u_n))$ iff $-1\in \sigma((w(u_1,...,u_n))^*)$, we may assume $\epsilon_1=1$. Let $j$ and $k$ denote the multiplicity of $x_{i_1}$ at the end and beginning of $w$, respectively, and assume $j\leq k$.
Now, write $ w=x_{i_1}^j x_{i_{\ell-j}}^{\epsilon_{\ell-j}}\cdot\cdot\cdot x_{i_{k+1}}^{\epsilon_{k+1}}x_{i_1}^k$ where $x_{i_{\ell-j}}\neq x_{i_1}\neq x_{i_{k+1}}$. Then, $$\sigma (w(u_1,...,u_n))=\sigma(u_{i_1}^{-j} w(u_1,...,u_n)u_{i_1}^j).$$ Again, the length of the word is unchanged. Hence, we assume $i_1\neq i_\ell$.

Now, define a map $\phi:\mathbb{F}_n\to \mathcal{S}_{2\ell}$ by mapping the generators $x_i$ to permutations $\sigma_i$ such that for each $1\leq k\leq \ell$ we require that

\begin{align}\label{ell}
    \sigma_{i_k}(k)=k+1, &\ \text{if}\ \epsilon_k=1 \nonumber\\
    \text{and}\\
    \sigma_{i_k}^{-1}(k)=k+1, &\ \text{if}\ \epsilon_k=-1.\nonumber
\end{align}
Note that, since $\ell\neq 1$ and $i_1\neq i_\ell$, the values for $\sigma_{i_1}^{\epsilon_1}(\ell+1)$ and $\sigma_{i_\ell}^{-\epsilon_\ell}(1)$ are not determined by the conditions in (\ref{ell}), which means $\phi( w)^{-1}(1)$ and $\phi( w)(\ell+1)$ are not determined by the conditions in (\ref{ell}). Hence, we are free to require for $1\leq k\leq \ell-1$, that
\begin{align}\label{ell+1}
    \sigma_{i_k}(\ell+k)=l+k+1, &\ \text{if}\ \epsilon_k=1\nonumber\\
    \text{and}\\
    \sigma_{i_k}^{-1}(\ell+k)=l+k+1, &\ \text{if}\ \epsilon_k=-1;\nonumber
\end{align}
and for $\sigma_{i_\ell}$ that
\begin{align}\label{2ell}
    \sigma_{i_\ell}(2\ell)=1, &\ \text{if}\ \epsilon_\ell=1\nonumber\\
    \text{and}\\
    \sigma_{i_\ell}(1)=2\ell, &\ \text{if}\ \epsilon_\ell=-1.\nonumber
\end{align}
Aside from these restrictions, the $\sigma_{i_k}^{\epsilon_k}$ are free to take any values. The following table provides an illustration of the enforced mappings.

\begin{center}
\begin{tabular}{cccccccccccccccccc}
& 1&& 2 && ... && $\ell$ && $\ell+1$ &&$\ell+2$ &&... && $2\ell$ && 1\\
$\sigma_{i_1}^{\epsilon_1}$ &&$\searrow$ &&&&&&&&$\searrow$&&&&&&&\\
& 1&& 2 && ... && $\ell$ && $\ell+1$ &&$\ell+2$ &&... && $2\ell$ && 1\\
$\sigma_{i_2}^{\epsilon_2}$ &&&&$\searrow$ &&&&&&&&$\searrow$ &&&&&\\
$\vdots$ &&&&&&&&&&&&&&&&&\\
& 1&& 2 && ... && $\ell$ && $\ell+1$ &&$\ell+2$ &&... && $2\ell$ && 1\\
$\sigma_{i_\ell}^{\epsilon_\ell}$ &&&&&&&&$\searrow$&&&&&&&&$\searrow$&\\
& 1&& 2 && ... && $\ell$ && $\ell+1$ &&$\ell+2$ &&... && $2\ell$ && 1\\
\end{tabular}
\end{center}

The conditions in (\ref{ell}) guarantee that $\phi( w)=\sigma_{i_\ell}^{\epsilon_\ell}...\sigma_{i_1}^{\epsilon_1}$ will map $1\mapsto \ell+1$.  The conditions in (\ref{ell+1}) and (\ref{2ell}) guarantee that $\phi( w)$ maps $\ell+1\mapsto 1$.
Hence, $\phi( w)$ has the two-cycle $(1,\ell+1)$ as desired. 
\end{proof}
 
\begin{remark}
 This proof likely uses too much ``space," and the argument may still work with $\mathcal{U}(l+1)$ instead. One would have to be very careful with choosing permutations with respect to the given word.
\end{remark}

We are now ready to prove Theorem \ref{eigenvalue}.

\begin{proof}[Proof of Theorem \ref{eigenvalue}.]
Let $w\in \mathbb{F}_n$ be a nontrivial, balanced, reduced word of length $\ell$, and let $\lambda\in\mathbb{T}$. By Lemma \ref{-1}, there exists a representation $\pi:C^*(\mathbb{F}_n)\to \mathbb{M}_{2\ell}$ for which $-1\in\sigma(\pi(w))$. By Lemma \ref{Sigma_d}, there then exists a representation \\
$\pi':C^*(\mathbb{F}_n)\to \mathbb{M}_{2\ell}$ with $\lambda\in\sigma(\pi'(w))$. 
\end{proof}

\bibliographystyle{amsplain}

\begin{thebibliography}{10}

\bibitem{AP} 
C. A. Akemann and G. K. Pedersen, Ideal Perturbation of Elements in $C^*$-algebras, Math. Scand., Vol 41 (1977).

\bibitem{Ar} R. Archbold, On residualy finite-dimensional $C^*$-algebras, 
Proc. Amer. Math. Soc., Vol. 123, Number 9, 1995

\bibitem{A}
W. Arveson, \textit{An Invitation to $C^\ast$-algebras}, Grad. Texts Math. \textbf{39}, Springer, New York-Heidelberg (1976).


\bibitem {Bekka}B. Bekka, Operator superrigidity for $SL_{n}(\mathbb{Z}), n
\ge3$, Invent. Math., 169 (2007), no.2, 401 -- 425.



\bibitem{BO}N. P. Brown and N. Ozawa, $C^*$-algebras and finite-dimensional
approximations, Graduate Studies in Mathematics, vol. 88, American Mathematical
Society, Providence, RI, 2008.

\bibitem{B1} B. Blackadar, \textit{Shape theorey for C*-algebras}; Math. Scand. (1985)
Vol. 56: 249-275.


\bibitem{B}
B. Blackadar, \textit{Operator algebras, Theory of $C^\ast$-algebras and von Neumann algebras}; Operator Algebras and Non-commutative Geometry, III, Encyclopaedia of Mathematical Sciences, vol. 122, Springer-Verlag, Berlin, 2006.

\bibitem{BK}
B. Blackadar, E. Kirchberg, Inner quasidiagonality and strong NF algebras. Pacific J. Math., 198(2):307–329, 2001.

\bibitem{C}
M. D. Choi, The full $C^\ast$-algebra of the free group on two generators, Pacific J. Math. \textbf{87} (1980), no. 1, 41–48.

\bibitem{Davidson} 
K. Davidson, C*-algebras by example, Fields Institute
Monograph 6, AMS, 1996.

\bibitem{EE}
S. Eilers and R. Exel, 
 Finite-dimensional representations of the soft torus. Proc. Amer. Math. Soc. 130 (2002), no. 3, 727 -– 731.

\bibitem{EL} R. Exel and T. Loring, Finite-dimensional representations of free product $C^*$-algebras, Int. J. Math, Vol. 03, Issue 04, 1992.

\bibitem{D}
J. Dixmier, \textit{$C^\ast$-algebras}, North Holland Publ Co., Amsterdam, 1977.

\bibitem{FNT}
T. Fritz, T. Netzer, A. Thom, Can you compute the operator norm? Proc. Amer. Math. Soc. \textbf{142} (2014), 4265-4276. 

\bibitem{G}
 J. Glimm, Type I $C^\ast$-algebras. Ann. Math. \textbf{73} (1961), 572-612.
 
 \bibitem{GM} K. R. Goodearl and P. Menal, Free and residually finite-dimensional C*-algebras,
J. Funct. Anal. 90(1990), 391-410.
 
\bibitem{GMR}
R. Grigorchuk, M. Musat, M. R\o rdam, Just-infinite $C^\ast$-algebras, arXiv:1604.08774 [math.OA], 2017.

\bibitem {Don} D. Hadwin, A lifting characterization of RFD C*-algebras,
Math. Scand. 115 (2014), no. 1, 85 -- 95.

\bibitem{HadwinShulman}
D. Hadwin and T. Shulman, Stability of group relations under small Hilbert-Schmidt perturbations, preprint. 


\bibitem{Korchagin}
A. Korchagin, Amalgamated free products of commutative $C^*$-algebras are residually finite-dimensional, J.Operator Theory 71(2), 2012


\bibitem{L}
T. Loring, \textit{Lifting solutions to perturbing problems in $C^\ast$-algebras}, volume 8 of Fields Institute Monographs, American Mathematical Society, Providence, RI, 1997.
 
\bibitem{LP}
T. Loring and G.K. Pedersen, Projectivity, transitivity, and AF-telescopes, Trans. Amer. Math. Soc. \textbf{350} (1998), 4313-4339. 

\bibitem{LoringShulman} 
T. Loring and T. Shulman, Lifting algebraic contractions in $C^*$-algebras, Operator Theory:
Advances and Applications, Vol. 233, 85 -– 92.

\bibitem{Shalom}
A. Lubotzky and Y. Shalom, Finite representations in the unitary dual and Ramanujan groups,
Discrete geometric analysis: proceedings of the first JAMS Symposium on Discrete Geometric
Analysis, 2002, Sendai, Japan, Contemporary Mathematics 347 (2004).

\bibitem{M}
C. C. Moore, Groups with finite-dimensional irreducible representations. Trans. Amer. Math. Soc. \textbf{166} (1972), 401–410.



\bibitem{S1}
S. Sakai, $C^\ast$-algebras and $W^\ast$-algebras. Classics in Mathematics. Springer-Verlag, Berlin, 1971. 1998 Reprint of 1971 Original.

\bibitem{S2}
S. Sakai, A characterisation of type I $C^\ast$-algebras. Bull. Amer. Math. Soc. \textbf{72} (1966), 508-512.

\bibitem{Thom}
A. Thom, Convergent sequences in discrete groups. The Canadian Mathematical Bulletin 56 (2013), 424–433.

\bibitem{T1}
E. Thoma, {\"U}ber unit{\"a}re Darstellungen abz{\"a}hlbarer diskreter Gruppen, Math. Ann. \textbf{153} (1964), 111-138. MR 28 \#3332.

\bibitem{T2}
E. Thoma, Ein Charakterisierung diskreter Gruppen vom Typ I, Invent. Math. \textbf{6} (1968), 190-196. MR 40 \#1540.

\end{thebibliography}

\end{document}